\documentclass[12pt]{article}
\usepackage{amsfonts}
\newtheorem{theorem}{Theorem}[section]
\newtheorem{cor}[theorem]{Corollary}
\newtheorem{prop}[theorem]{Proposition}
\newtheorem{lemma}[theorem]{Lemma}
\newcommand{\qed}{\qquad$\Box$}
\newtheorem{unnumber}{}

\newtheorem{prepf}[unnumber]{Proof}
\newenvironment{proof}{\prepf\rm}{\endprepf}

\newcommand{\Aut}{\mathop{\mathrm{Aut}}}
\begin{document}

\title{A footnote to the KPT theorem in structural Ramsey theory}
\author{Peter J. Cameron\footnote{Corresponding author,
\texttt{pjc20@st-andrews.ac.uk}}\ \ 
and Siavash
Lashkarighouchani\footnote{\texttt{siavash.lashkarighouchani@gmail.com}}\\
University of St Andrews}
\date{}
\maketitle

\begin{abstract}
The celebrated theorem of Kechris, Pestov and Todor\v{c}evi\'c connecting
structural Ramsey theory with topological dynamics has as a consequence that
the Fra\"{\i}ss\'e limit of a Ramsey class of non-trivial finite relational
structures has a reduct which is a total order; this implies an earlier result 
of Ne\v{s}et\v{r}il, according to which the structures in such a class are
rigid (have trivial automorphism groups). In this paper, we give an alternative
proof of this fact. If $\mathcal{C}$ is a Fra\"{\i}ss\'e class of rigid
structures over a finite relational language, then either the Fra\"{\i}ss\'e
limit of $\mathcal{C}$ has a reduct which is a total order, or there is an
explicit failure of the Ramsey property involving a pair $(A,B)$ of structures
in $\mathcal{C}$ with $|A|=2$.


\end{abstract}

\section{Introduction}

It is nearly a century since Ramsey's theorem was published. The theorem states
that if $a$ and $b$ are positive integers, then there exists $c$ such that, if
the $a$-element subsets of a $c$-element set are coloured with two colours
(red and blue, say), then there exists a $b$-element set which is
\emph{monochromatic} (all of its $a$-element subsets have the same colour).

Since its publication, the work of many mathematicians, including Paul
Erd\H{o}s and his collaborators, have extended it from a theorem to a subject
(Ramsey theory) with links to logic, set theory, combinatorics, group theory,
number theory, topological dynamics, and other areas.

In particular, structural Ramsey theory deals with classes $\mathcal{C}$ of
finite structures over finite relational languages. Such a structure 
consists a set carrying named relations with prescribed arities. We always
assume that our classes are closed under isomorphism and \emph{hereditary}
(closed under taking induced substructures), and that all structures are
non-empty. We say that $\mathcal{C}$ is a \emph{Ramsey class} if, given
any two structures $A,B\in\mathcal{C}$, there exists $C\in\mathcal{C}$ such
that, if the embeddings $A\to C$ are coloured red and blue, then there is an
embedding $B\to C$ which is monochromatic, in the sense that all the
embeddings whose image is contained in the image of $B$ have the same colour.

A closely related notion is that of a \emph{Fra\"{\i}ss\'e class}, a hereditary
class $\mathcal{C}$ of finite structures satisfying
\begin{enumerate}
\item the \emph{joint embedding property}: given $B_1,B_2\in\mathcal{C}$, there
exists $C\in\mathcal{C}$ such that $B_1$ and $B_2$ are embeddable in $C$; and
\item the \emph{amalgamation property}: given $A,B_1,B_2\in\mathcal{C}$ and
embeddings $f_i:A\to B_i$ for $i=1,2$, there is a structure $C\in\mathcal{C}$
and embeddings $g_i:B_i\to C$ such that $f_1g_1=f_2g_2$.
\end{enumerate}
We will also use the \emph{strong amalgamation property}, which asserts
that in (b), the embeddings can be chosen so that the images of $g_1$ and
$g_2$ intersect precisely in the image of $f$.

The \emph{age} of a relational structure is the class of structures embeddable
in it. A structure $M$ is \emph{homogeneous} if every isomorphism between finite
substructures can be extended to an automorphism of $M$. Fra\"{\i}ss\'e's
theorem states:

\begin{theorem}
A class $\mathcal{C}$ of finite structures over a finite relational language,
closed under isomorphism and substructures, is the age of a finite or countable
homogeneous structure $M$ if and only if $\mathcal{C}$ satisfies the joint
embedding and amalgamation properties. If it does, then $M$ is unique up to
isomorphism.
\end{theorem}

In this situation, $M$ is called the \emph{Fra\"{\i}ss\'e limit} of
$\mathcal{C}$.

Ramsey classes were studied by Ne\v{s}et\v{r}il~\cite{nesetril1,nesetril2},
who showed:

\begin{theorem}
\begin{enumerate}
\item A Ramsey class over a finite language is a Fra\"{\i}ss\'e class
satisfying the strong amalgamation property.
\item If $\mathcal{C}$ is a nontrivial Ramsey class (that is, not consisting
of unstructured sets), then the objects in $\mathcal{C}$ are rigid (they have
trivial automorphism groups).
\end{enumerate}
\end{theorem}

Many examples of Ramsey classes are given in \cite{hn}.

The second part of this theorem was subsequently explained by the celebrated
theorem of Kechris, Pestov and Todor\v{c}evi\'c~\cite{kpt}. First, some
background. The symmetric group on a countable set has a natural topology, the
topology of pointwise convergence: a basis for the topology consists of the
cosets of stabilisers of finite tuples. So any subgroup has the induced
subspace topology. It is known that a subgroup is closed if and only if it is
the automorphism group of a relational structure.

A topological group $G$ is \emph{extremely amenable} if any continuous action
on a compact space has a fixed point. Now the KPT theorem states:

\begin{theorem}
Let $\mathcal{C}$ be a nontrivial Ramsey class with Fra\"{\i}ss\'e limit $M$.
Then the automorphism group of $M$ is extremely amenable.
\end{theorem}

Now the set of all total orders on a countable set is compact, and the
symmetric group acts continuously on it. For
completeness, we give the proof. Take a countable set, which we can assume to
be $\mathbb{N}$. Let $\mathcal{L}$ be the set of linear (total) orders on
$\mathbb{N}$. There is a topology on $\mathcal{L}$, defined as follows:
A basis for the open sets consists of all the sets for which
$\{0,1,\ldots,n-1\}$ have a given ordering, for all $n$.

\begin{lemma}
\begin{enumerate}
\item The topology on $\mathcal{L}$ is compact.
\item The symmetric group on $\mathbb{N}$ acts continuously on
$\mathcal{L}$.
\end{enumerate}
\end{lemma}

\begin{proof} (a) We can identify the topology on $\mathcal{L}$ with the
product topology on $A_1\times A_2\times\cdots$, with $|A_i|=i$ and with the
discrete topology on each set $A_i$.

This works as follows: if we have ordered $0,1,\ldots,n-1$, then there are 
$n$ gaps into which $n$ can be put: before the first element, between the
first and second, \dots, after the last. Thus choosing elements of 
$A_1\times\cdots\times A_n$ corresponds to ordering $1,2,\ldots,n$. These
choices give a basis of open sets for the product topology.

Now Tychonoff's theorem shows that the product (and hence also $\mathcal{L}$)
is compact.

\smallskip

(b) Choose a permutation $f$ of $\mathbb{N}$. Let $m-1$ be the largest element
in $f^{-1}\{0,1,\ldots,n-1\}$. Then the inverse image of the basic open set
defined by the ordering of $0,1,\ldots,n-1$ is a union of finitely many basic
open sets defined by orderings of $0,1,\ldots,m-1$, and so is open. So $f$ is
continuous. \qed
\end{proof}

So, if $\mathcal{C}$ is a Ramsey class with Fra\"{\i}ss\'e limit $M$, then
$\Aut(M)$ preserves a total order $<$ on $M$. It follows that every structure
$A$ in $\mathcal{C}$ carries the total order induced from $<$, which
is preserved by its automorphisms, since these automorphisms extend to
automorphisms of $A$. Hence $A$ is rigid.

It follows that, if a Fra\"{\i}ss\'e class is made up of rigid structures, then
either its elements carry invariant total orders, or the class fails to be a
Ramsey class. Our purpose in the paper is to give a direct constructive proof
of this theorem. In the following result, readers may interpret the term
``reduct'' of $M$ to mean a structure invariant under the automorphisms of $M$;
we will explain further in the next section. 

\begin{theorem}
Let $\mathcal{C}$ be a Fra\"{\i}ss\'e class of structures over a finite
relational language. Assume that the elements of $\mathcal{C}$ are rigid. Then
one of the following holds:
\begin{enumerate}
\item the Fra\"{\i}ss\'e limit of $\mathcal{C}$ has a reduct which is a total
order;
\item there are structures $A,B\in\mathcal{C}$ with $|A|=2$
which demonstrate that $\mathcal{C}$ is not a Ramsey class.
\end{enumerate}\label{t:main}
\end{theorem}

\begin{cor}
With the hypotheses of the Theorem, if $\mathcal{C}$ has a unique $2$-element
structure up to isomorphism, then in the second conclusion of the theorem, we
have $|A|=2$ and $|B|=3$.
\label{c:main}
\end{cor}

The structure of the rest of the paper is as follows. The next section will
contain some preliminaries from model theory which underlie our arguments.
Then we describe a simple example of a Fra\"{\i}ss\'e class of rigid structures
and give an explicit failure of the Ramsey property. Next we prove the theorem,
and we conclude with some comments and further questions.

\section{Logical preliminaries}

We will use the theorem of Engeler, Ryll-Nardzewski and Svenonius
\cite[Theorem 6.3.1]{hodges}. In this section we describe the theorem and
its background. The context is first-order logic, though we only use a
specialisation: we work over a relational language, containing no constant
or function symbols, but the theorem holds in general. We refer to Wilfrid
Hodges' book~\cite{hodges} for further background information.

A permutation group $G$ on a set $X$ is \emph{oligomorphic} if its action
on $X^n$ has only finitely many orbits for every positive integer $n$.

A \emph{theory} is simply a set of first-order sentences (formulae with no free
variables); the theory of a structure $M$ is the set of all
first-order sentences which hold in $M$. Dually, a model of a theory is a
structure in which all sentences of the theory hold.
If $T$ is a theory, then a \emph{type} is a set of formulae with $n$ free
variables which is maximal with respect to being consistent with $T$. A type
is \emph{realised} in a model $M$ for $T$ if there are $n$ elements of $M$
which can be substituted for the free variables so as to make all the formulae
true.

A theory is \emph{$\omega$-categorical} if it has, up to isomorphism, a unique
countable model. By slight abuse of language, we say that a structure $M$ is
$\omega$-categorical if its theory is; we treat an $\omega$-categorical theory
and its countable model as being essentially equivalent. Now the theorem of
Engeler, Ryll-Nardzewski and Svenonius from 1959 says the following:

\begin{theorem}
Let $M$ be a first-order structure with theory $T$. Then the following
are equivalent:
\begin{enumerate}
\item $T$ is $\omega$-categorical;
\item for every natural number $n$, there are only finitely many $n$-types
over $T$, all of which are realised in $M$;
\item the automorphism group $\Aut(M)$ is oligomorphic.
\end{enumerate}
If these conditions hold, then the set of $n$-tuples of $M$ realising a given
$n$-type is an orbit of $\Aut(M)$ on $M^n$, and every orbit corresponds to an
$n$-type in this way.
\label{t:erns}
\end{theorem}

It follows that a homogeneous relational structure $M$ over a finite relational
language is $\omega$-categorical. The orbits on $n$-tuples of distinct elements
correspond to the $n$-element structures in the age of $M$.

In general, a \emph{reduct} of a structure $M$ consists of a set of relations
(possibly over a different language) which are definable in terms of the
relations of $M$ by first-order formula without free variables. If $M$ is
$\omega$-categorical, then it follows from Theorem~\ref{t:erns} that reducts
of $M$ correspond bijectively to closed overgroups of $\Aut(M)$ in the
symmetric group.

If $M$ is homogeneous over a finite relational language,
and $A$ is a finite substructure of $M$, then the restriction to $A$ of
the relations in any reduct of $M$ are preserved by any automorphism of $A$,
since the automorphisms of $A$ are induced by automorphisms of $M$ fixing $A$
setwise. This explains why the existence of a total order as a reduct of $M$
implies that all structures in the age of $M$ are rigid.

\section{An example}
\label{s:ex}

This example is taken from the second author's thesis~\cite{siavash}.

Here is an elementary group-theoretic result we use: \emph{Cauchy's theorem}.

\begin{theorem}
If a finite group $G$ has order divisible by a prime $p$, then $G$ contains an
element of order $p$.
\label{t:cauchy}
\end{theorem}

A \emph{tournament} is a set with a binary operation $\to$ on $X$ such that
\begin{enumerate} 
\item $\neg(x\to x)$ for all $x\in X$;
\item for any $x,y\in X$ with $x\ne y$, exactly one of $x\to y$ and $y\to x$
holds.
\end{enumerate}

\begin{prop}
\begin{enumerate}
\item Finite tournaments form a Fra\"{\i}ss\'e class with the strong
amalgamation property.
\item Finite tournaments have automorphism groups of odd order.
\end{enumerate}
\end{prop}

\begin{proof} (a) is straightforward. For (b), if $T$ is a tournament with
$|\Aut(T)|$ even, then by Theorem~\ref{t:cauchy} there would be an automorphism
of order~$2$, which necessarily interchanges two points; this is not possible.
\qed
\end{proof}

A \emph{C-structure} is a ternary relational structure defined on the
set of leaves of a finite rooted binary tree as follows: given three leaves
$x,y,z$, one of them, say $x$, is distinguished by the fact that the paths from
the root to $y$ and to $z$ coincide for longer than either does with the path
from the root to $x$. In other words, omitting vertices not on these paths and
suppressing divalent vertices, the tree looks like this:

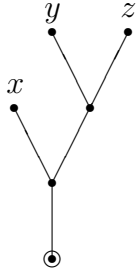
\begin{center}
\setlength{\unitlength}{1mm}
\begin{picture}(20,35)
\put(10,0){\circle*{1}} \put(10,0){\circle{2}}
\multiput(10,10)(5,10){3}{\circle*{1}}
\multiput(5,20)(5,10){2}{\circle*{1}}
\put(10,0){\line(0,1){10}}
\put(10,10){\line(1,2){10}}
\multiput(10,10)(5,10){2}{\line(-1,2){5}}
\put(4,22){$x$}
\put(9,32){$y$}
\put(19,32){$z$}
\end{picture}
\end{center}
We write the relation as $R(x\mid yz)$.

\begin{prop}
\begin{enumerate}
\item The finite C-relations form a Fra\"{\i}ss\'e class with the strong
amalgamation property.
\item The automorphism group of a finite C-relation has order a power of $2$.
\end{enumerate}
\end{prop}

\begin{proof} (a) This is proved by showing that the C-relation determines the
tree; we refer to Adeleke and Neumann~\cite{an} for details.

\smallskip

(b) In a counterexample, Theorem~\ref{t:cauchy} shows that there would be an
automorphism of odd prime order $p$. Such an automorpism preserves the tree,
and fixes the root, so fixes the two children of the root; working up the
tree, we see that it is the identity, a contradiction.
\end{proof}

Now we take $\mathcal{C}$ to be the class of structures consisting of a finite
set carrying both a tournament and a C-relation, these being chosen
independently. Since both tournaments and C-relations are Fra\"{\i}ss\'e 
classes with strong amalgamation, so is $\mathcal{C}$. Moreover, the
automorphism group of a $\mathcal{C}$-structure is the intersection of the
automorphism group of a tournament (of odd order) with that of a C-structure
(of $2$-power order), and so is trivial.

Now we show that $\mathcal{C}$ is not a Ramsey class.
Since the Ramsey property has several alternations of
quantifiers, we state explicitly what it means for a class $\mathcal{C}$ to
fail the Ramsey property: there exist $A,B\in\mathcal{C}$ such that, for any
$C\in\mathcal{C}$, we can colour the embeddings of $A$ into $C$ so that there
is no monochromatic embedding of $B$.

We take $A$ to be the unique $2$-element structure in $\mathcal{C}$ (a single
tournament arc $x\to y$), and $B$ the unique $3$-element structure on which the
tournament is a $3$-cycle. Take any structure $C\in\mathcal{C}$. We colour
the embeddings of $A$ into $C$ (that is, the tournament arcs in $C$), as
follows. Choose an arbitrary total order $<$ on $C$, and colour $u\to v$ red
if $u<v$ and blue if $u>v$. Now a copy of $B$ in $C$ cannot be monochromatic:
if all three arcs were red, it would agree with the total order $<$, and if
they were blue, it would agree with the converse of $<$; both possibilities
contradict the fact that $B$ carries a $3$-cycle, not a total order.

\section{Proofs}

Before beginning the proof, we make an observation that will be used often.
A directed graph $D$ is \emph{acyclic} if it contains no directed cycle. Now
if $D$ is an acyclic digraph, then there is a total order $<$ on the vertices
of $D$ such that $x\to y$ implies $x<y$. To see this, take the set $X_1$ of
\emph{sinks} in $D$ (the vertices with no incoming arcs) to be the minimal
elemennts, and if $X_1,\ldots,X_i$ have been defined and their union is not
$D$, take $D_{i+1}$ to be the set of sinks of
$D\setminus(X_1\cup\cdots\cup X_i)$. Then construct the order by putting an
arbitrary total order on each set $X_i$ (noting that that there are no arcs
within $X_i$); and, if $x\in X_i$ and $y\in X_j$ with $i<j$, putting $x<y$.

\paragraph{Proof of Theorem~\ref{t:main} and Corollary~\ref{c:main}}
We treat first the simpler case where $\mathcal{C}$ contains a unique
$2$-element structure up to isomorphism; we verify that the conclusions of the
Corollary hold in this case.

Let $M$ be the Fra\"{\i}ss\'e limit of $\mathcal{C}$. Then
since the language is finite, $M$ is $\omega$-categorical, so $n$-types in
the theory of $M$ coincide with orbits of that automorphism group of $M$ on
$n$-tuples, by the Ryll-Nardzewski theorem.

Let $\{a,b\}$ be a $2$-element structure in $\mathcal{C}$. Since it is
rigid, the ordered pairs $(a,b)$ and $(b,a)$ have different types, and so
there is a first-order formula which distinguishes them. This formula
defines an ordering of the pair $(a,b)$: we put $a\to b$ if it holds, and
$b\to a$ otherwise. This is clearly a tournament which is a reduct of $M$.

If $T$ is a total order, then we have the first alternative of the Theorem;
so suppose not. Then $T$ fails to satisfy transitivity, so it contains a
cyclic triple $a\to b\to c\to a$. Now take $A$ to be the unique $2$-element
structure in $\mathcal{C}$, and $B$ the induced substructure on a cyclic
triple of $T$.

Let $C$ be any member of $\mathcal{C}$. We choose an arbitrary total order on
$C$, and define a colouring of the $2$-element substructures by the rule that
$\{a,b\}$ is red if $(a\to b)\Leftrightarrow(a<b)$ (that is, the total order
agrees with the tournament on $\{a,b\}$) and blue otherwise. Now
$C$ contains no monochromatic copy of $B$: for if all the $2$-subsets of $B$
were red, then the tournament would agree with the total order on $B$, while
if they were blue, it would agree with its converse, and neither of these can
be the case.

\medskip

Now we turn to the general case, where there are finitely many $2$-element
structures in $\mathcal{C}$ up to isomorphism, say $A_1,A_2,\ldots,A_s$.
In any structure in $\mathcal{C}$, or in its Fra\"{\i}ss\'e limit $M$,
the pairs carrying an element of $A_i$ form a graph $\Gamma_i$; these graphs
partition the complete graph. Also, by Ryll-Nardzewski, there is a first-order
formula which defines a direction on each edge of $\Gamma_i$, hence defining
a directed graph $D_i$ and its converse $D_i^*$.

We claim that, if any $D_i$ contains a directed cycle, then we have a failure
of the Ramsey property. For let $A$ be the structure $A_i$, and $B$ the
structure induced on the vertex set of the directed cycle. Let $C$ be any
element of $\mathcal{C}$. We define a $2$-colouring of the embedded copies of
$A$ in $C$ as follows. Choose an arbitrary total order $<$ on $C$, and colour
a pair $\{x,y\}$ carrying a copy of $A_i$ red if
$((x\to y)\Leftrightarrow(x<y))$ and blue otherwise.  In any embedded copy
of $B$, the pairs in the directed cycle cannot all be red, or the order $<$
would contain a directed cycle, and similarly cannot all be blue.

Now we show that we can choose one from each pair $(D_i,D_i^*)$ and call it
$D_i$ so that, if $E_i=D_1\cup\cdots\cup D_i$, then either the digraph
$E_i$ is acyclic, or we have a failure of the Ramsey property. This is proved
by induction on $i$; we have already observed the case $i=1$. So suppose
that $D_1,\ldots,D_i$ are defined, and let $\{x,y\}$ be a $2$-set inducing
a copy of $A_{i+1}$.

We subdivide into three cases, according as there are directed paths in $E_i$
joining both, just one, or neither of $(x,y)$ and $(y,x)$.

\paragraph{Case 1:} There is a path from $x$ to $y$ and a path from $y$ to $x$.
Concatenating these paths, we find a directed cycle in $E_i$, contrary to
hypothesis.

\paragraph{Case 2:} There is a path from $x$ to $y$ but not from $y$ to $x$.
Let $D_{i+1}$ be the digraph having an arc $x\to y$, and put
$E_{i+1}=E_i\cup D_{i+1}$. Suppose that $E_{i+1}$ contains a directed cycle.
Any arc $u\to v$ of this cycle which belongs to $D_{i+1}$ can be replaced
by a directed path from $u$ to $v$ in $E_i$ (since these arcs form an orbit
of $\Aut(M)$, by Theorem~\ref{t:erns}). The result is a directed cycle
in $E_i$, again contrary to assumption.

\paragraph{Case 3:} There are no directed paths in $E_i$ between $x$ and $y$
(in either order). Set $E_{i+1}=E_i\cup D_{i+1}$ and
$E^*_{i+1}=E_i\cup D^*_{i+1}$. If either one of these digraphs is acyclic,
then we change the notation if necessary so that this one is $E_{i+1}$, and
the inductive step is complete.

Suppose that both of these digraphs contain
directed cycles (which we may assume include $x\to y$ or $y\to x$ as
appropriate). Note that each of these cycles must contain an arc of $D_{i+1}$
or $D^*_{i+1}$ other than $x\to y$ or $y\to x$, by the case assumption.
Now let $A=A_{i+1}$ and let $B$ be the structure induced on the
union of these two cycles. We claim the Ramsey property fails for $A$ and $B$.

So let $C$ be any structure in $\mathcal{C}$. We know that the restriction of
$E_i$ to $C$ is acyclic, so choose a total order $<$ on $C$ so that $u<v$ for
any  arc $u\to v$ in $E_i$. Colour a pair $\{x',y'\}$ inducing $A_{i+1}$ red
if $x'<y'$ and blue otherwise. We claim that no copy of $B$ can be
monochromatic. To see this, suppose first that $x<y$. The directed cycle in
$E_{i+1}$ containing the arc $x\to y$ has all its arcs in $E_i$ red, as well
as the arc $x\to y$; but not all its arcs can be red, so there must be another
arc of $D_{i+1}$ in this cycle which is coloured blue. A similar argument
holds in the other case. This concludes the discussion of Case 3.

\medskip

After having dealt with all of the types of $2$-element structures, either
we have found a failure of Ramsey's Theorem, or we have built an acyclic
directed graph on $M$ in which every pair of vertices are connected by an arc.
This digraph is a reduct of $M$, since each digraph $D_i$ is defined by a
first-order formula and there are only finitely many of them. So we have
found a total order which is a reduct of $M$, and the proof of the theorem
is complete. \qed

\section{Further comments}

We mention here some questions that might be of interest.
\begin{enumerate}
\item Our results hint that there might be a bound for $|B|$ (if the pair
$(A,B)$ fails the Ramsey property) in terms of the number of isomorphism
types of $2$-element structures. Is this true?
\item The Fra\"{\i}ss\'e limit of the structure constructed in
Section~\ref{s:ex} has automorphism group which is torsion-free but not
extremely amenable. This group could be investigated further. In particular,
does it have any properties related to amenability? Is it simple? What are
its reducts?
\item Is there a general theory of Fra\"{\i}ss\'e classes of rigid structures?
\end{enumerate}

\end{document}